\documentclass[final]{amsart}
\usepackage{amsmath}
\usepackage{amssymb}
\usepackage{amsmath, amssymb, amsfonts, amsxtra}
\theoremstyle{proclaim}
\newtheorem{theorem}{Theorem}[section]
\newtheorem{lemma}[theorem]{Lemma}
\newtheorem{corollary}[theorem]{Corollary}

\theoremstyle{fancyproclaim}

\theoremstyle{statement}
\newtheorem{remark}[theorem]{Remark}

\newtheorem{example}[theorem]{Example}
\theoremstyle{fancystatement}

\numberwithin{equation}{section}

\providecommand{\AMS}{$\mathcal{A}$\kern-.1667em%
	\lower.25em\hbox{$\mathcal{M}$}\kern-.125em$\mathcal{S}$}
\begin{document}
	\title[Characterization of $k-$smooth operators between Banach spaces ]{Characterization of $k-$smooth operators between Banach spaces}
	\author[Arpita Mal and Kallol Paul ]{Arpita Mal and Kallol Paul}

	\newcommand{\acr}{\newline\indent}

	\address[Mal]{Department of Mathematics\\ Jadavpur University\\ Kolkata 700032\\ West Bengal\\ INDIA}
	\email{arpitamalju@gmail.com}
	
	\address[Paul]{Department of Mathematics\\ Jadavpur University\\ Kolkata 700032\\ West Bengal\\ INDIA}
	\email{kalloldada@gmail.com}

	\thanks{The research of  Arpita Mal is supported by UGC, Govt. of India.  The research of Prof. Paul  is supported by project MATRICS(MTR/2017/000059)  of DST, Govt. of India. } 
	

\maketitle
\begin{abstract}
 We study $k-$smoothness of bounded linear operators  defined between arbitrary  Banach spaces. As an application, we characterize $k-$smooth operators defined from $\ell_1^n$ to an arbitrary Banach space. We also completely characterize $k-$smooth operators defined between arbitrary two-dimensional Banach spaces.  
\end{abstract}

\section{Introduction}
The characterization of smoothness of operator between Banach spaces is a rich, intricate problem to study. It helps to understand the geometry of operator space. Over the years several mathematicians have been studying the smoothness of operators defined between Banach spaces. The readers may go through \cite{DeK,GY,HR,KY,MPRS,PSG,R,Ra,SPM,SPMR} for more results in this direction. Before proceeding further, we introduce the notations and terminologies to be used throughout the paper.

The letters $\mathbb{X},\mathbb{Y}$ denote real Banach spaces. The unit ball, unit sphere and the dual space of $\mathbb{X}$ are denoted respectively by $B_{\mathbb{X}}=\{x\in \mathbb{X}:\|x\|\leq 1\},S_\mathbb{X}=\{x\in \mathbb{X}:\|x\|\leq 1\}$ and $\mathbb{X}^*.$ The set of all extreme points of $B_\mathbb{X}$ is denoted by $Ext(B_{\mathbb{X}}).$ For any set $A,$ $|A|$ denotes the cardinality of $A.$ The space of all bounded (compact) linear operators is denoted by $\mathbb{L}(\mathbb{X},\mathbb{Y})~(\mathbb{K}(\mathbb{X},\mathbb{Y})).$  An element $x^*\in S_{\mathbb{X}^*}$ is  said to be a supporting linear functional of $x\in S_{\mathbb{X}},$ if $x^*(x)=1.$ Suppose $J(x)$ denotes the set of all supporting linear functionals of $x,$ i.e., $J(x)=\{x^*\in S_{\mathbb{X}^*}:x^*(x)=1\}.$ Note that, $J(x)$ is a weak*-compact convex subset of $S_{\mathbb{X}^*}.$ The set of all extreme points of $J(x)$ is denoted by $Ext~J(x).$ An element $x\in S_{\mathbb{X}}$ is said to be smooth if $J(x)$ is singleton. So an interesting problem is to study the ``size" of $J(x),$ whenever $J(x)$ is not singleton. In $2005$, Khalil and Saleh \cite{KS} have turned their attention to this problem. In  \cite{KS} they have generalized the notion of smoothness and introduced the notion of $k-$smoothness or multi-smoothness. Following \cite{KS}, we say that an element $x\in S_{\mathbb{X}}$ is $k-$smooth or the order of smoothness of $x$ is $k,$ if $J(x)$ contains exactly $k$ linearly independent vectors, i.e., if $k=dim ~span ~J(x).$ Similarly, an operator $T\in \mathbb{L}(\mathbb{X},\mathbb{Y})$ is said to be $k-$smooth operator if $k=dim~span~J(T),$ i.e., if there exist exactly $k$ linearly independent functionals in $S_{\mathbb{L}(\mathbb{X},\mathbb{Y})^*}$ supporting the operator $T.$ In \cite{H,Ha,KS,LR,Wa}, the authors have extensively studied $k-$smoothness in Banach spaces and in operator spaces. Though the characterization of $k-$smooth operators defined on Hilbert spaces \cite{Wa} and between some particular Banach spaces are known, the complete characterization between arbitrary Banach spaces is still open. The main purpose of this paper is to proceed substantially in this direction. To do so we will use norm attainment set of an operator defined as :  For $T\in \mathbb{L}(\mathbb{X},\mathbb{Y}),$ the norm attainment set, denoted as $M_T$, is the collection of all unit vectors at which $T$ attains its norm, i.e.,  $M_T=\{x\in S_{\mathbb{X}}:\|Tx\|=\|T\|\}.$ To look into the properties of norm attainment set and its role in the study of smoothness of operators one may go through \cite{MPRS,PSG,S,SPMR}.

In this paper, we first characterize the order of smoothness of some class of operators defined between a finite dimensional Banach space and an arbitrary Banach space depending on the norm attainment sets of the operators. As a result, we can completely characterize $k-$smooth operators defined between $\ell_1^n$ and an arbitrary Banach space. Finally, we characterize the order of smoothness of $T\in \mathbb{L}(\mathbb{X},\mathbb{Y}),$ where $\mathbb{X},\mathbb{Y}$ are arbitrary two-dimensional Banach spaces. To obtain these results, we mainly use the following  lemma from \cite[Lemma 3.1]{W}, which characterizes $Ext~J(T)$ in terms of $Ext~J(Tx)$ and $M_T\cap Ext(B_{\mathbb{X}})\ni x$. 
\begin{lemma}\cite[Lemma 3.1]{W}\label{lemma-wojcik}
	Suppose that $\mathbb{X}$ is a reflexive Banach space. Suppose that $\mathbb{K}(\mathbb{X},\mathbb{Y})$ is an $M-$ideal in $\mathbb{L}(\mathbb{X},\mathbb{Y}).$ Let $T\in \mathbb{L}(\mathbb{X},\mathbb{Y}), \|T\|=1$ and  dist$(T,\mathbb{K}(\mathbb{X},\mathbb{Y}))<1.$  Then $M_T\cap Ext(B_\mathbb{X})\neq \emptyset$ and 
	\[Ext ~J(T)=\{y^*\otimes x\in \mathbb{K}(\mathbb{X},\mathbb{Y})^*:x\in M_T\cap Ext(B_{\mathbb{X}}), y^*\in Ext ~J(Tx)\},\]
	where  $y^*\otimes x: \mathbb{K}(\mathbb{X},\mathbb{Y})\to \mathbb{R}$ is defined by $y^*\otimes x(S)=y^*(Sx)$ for every $S\in \mathbb{K}(\mathbb{X},\mathbb{Y}).$
\end{lemma}

\section{Main results}
We begin this section with an easy Lemma which will be used later to prove some of the theorems of this section. The proof of the lemma being simple, we omit the proof here.
\begin{lemma}\label{lemma-01}
Suppose  $\mathbb{X},\mathbb{Y}$ are finite dimensional Banach spaces. If $\{x_1,x_2,\ldots,x_m\}$ is a linearly independent subset of $\mathbb{X}$ and $\{y_1^*,y_2^*,\ldots,y_n^*\}$  is a linearly independent subset of $\mathbb{Y}^*$ then $\{y_i^*\otimes x_j:1\leq i\leq n,1\leq j\leq m\}$ is a linearly independent subset of $\mathbb{L}(\mathbb{X},\mathbb{Y})^*.$  	
\end{lemma}

Observe that, if $\mathbb{X}$ is a finite dimensional Banach space, $\mathbb{Y}$  is arbitrary Banach space and if $T\in \mathbb{L}(\mathbb{X},\mathbb{Y})$ $(=\mathbb{K}(\mathbb{X},\mathbb{Y}))$ is such that $\|T\|=1$ holds, then $\mathbb{X},\mathbb{Y}$ and $T$ satisfies all the conditions of Lemma \ref{lemma-wojcik}. Using Lemma \ref{lemma-wojcik}, we now characterize the order of smoothness of a class of operators defined between a finite dimensional Banach space and an arbitrary Banach space.

\begin{theorem}\label{th-ind}
	Suppose $\mathbb{X}$ is a finite dimensional Banach space and $\mathbb{Y}$ is arbitrary Banach space. Suppose that $T\in \mathbb{L}(\mathbb{X},\mathbb{Y})$ is such that $\|T\|=1$ and  $M_T\cap Ext(B_{\mathbb{X}})=\{\pm x_1,\pm x_2,\ldots,\pm x_r\},$ where $\{x_1,x_2,\ldots,x_r\}$ is linearly independent in $\mathbb{X}.$ Then $T$ is $k-$smooth if and only if $Tx_i$ is $m_i-$smooth for each $1\leq i\leq r$ and $m_1+m_2+\ldots+m_r=k.$
\end{theorem}
\begin{proof}
	Let $dim(\mathbb{X})=n.$ At first suppose that $r<n.$ Extend $\{x_1,x_2,\ldots,x_r\}$ to a basis $\{x_1,x_2,\ldots,x_n\}$ of $\mathbb{X}.$ Suppose $T$ is $k-$smooth and $Tx_i$ is $m_i-$smooth for each $1\leq i\leq r.$  Then by \cite[Prop. 2.1]{LR}, for each $1\leq i\leq r,$ we have, 
		\begin{eqnarray*}
		m_i&=&dim ~span ~J(Tx_i)\\
		&=& dim~ span~ Ext~ J(Tx_i).
	\end{eqnarray*}

    Let $\{y_{ij}^*:1\leq j\leq m_i, y_{ij}^*\in Ext ~J(Tx_i)\}$ be a basis of $span~Ext~J(Tx_i)$ for each $1\leq i\leq r.$  Let 
     \[W_i=span ~\{y_{ij}^*\otimes x_i:y_{ij}^*\in Ext ~J(Tx_i)\} ~\text{for each } 1\leq i\leq r.\]
     We first show that $B_i=\{y_{ij}^*\otimes x_i:1\leq j\leq m_i\}$ is a basis of $W_i.$ 
     Let $\sum\limits_{1\leq j\leq m_i}a_j(y_{ij}^*\otimes x_i)=0,$ where $a_j\in \mathbb{R}$ for all $1\leq j\leq m_i.$ Consider a Hamel basis $\{u_\beta:\beta \in \Lambda\}$ of $\mathbb{Y}.$ For each $\beta\in\Lambda,$ define $S_\beta\in \mathbb{L}(\mathbb{X},\mathbb{Y})$ by 
     \begin{equation} \label{eq1}
     \begin{split}
    	S_\beta x_i&=u_\beta\\
    S_\beta x_l&= 0 \text{~~ for all } 1\leq l(\neq i)\leq n. 
    \end{split}
     \end{equation}
    Then for each $\beta \in \Lambda,$\\
    $\sum\limits_{1\leq j\leq m_i}a_j(y_{ij}^*\otimes x_i)(S_\beta)=0\Rightarrow\sum\limits_{1\leq j\leq m_i}a_jy_{ij}^*S_\beta(x_i)=0\Rightarrow \sum\limits_{1\leq j\leq m_i}a_jy_{ij}^*(u_\beta)=0\Rightarrow \sum\limits_{1\leq j\leq m_i}a_jy_{ij}^*=0\Rightarrow a_j=0 $ for all $1\leq j\leq m_i.$ Thus, $B_i$ is linearly independent. It can be easily verified that $B_i$ is a spanning set of $W_i.$ Hence, $B_i$ is a basis of $W_i$  and so $dim~W_i=m_i$ for each $1\leq i\leq r.$ Now,
	\begin{eqnarray*}
	k&=&dim ~span ~J(T)\\
	&=& dim~ span~ Ext ~J(T)\\
	&=& dim~ span ~\{y_{ij}^*\otimes x_i:y_{ij}^*\in Ext ~J(Tx_i),1\leq i\leq r \}\\
	&=& dim~W,\text{~where,}\\
	W&=&span ~\{y_{ij}^*\otimes x_i:y_{ij}^*\in Ext~ J(Tx_i),1\leq i\leq r \}.
	\end{eqnarray*}
    We now show that $W=\oplus_{i=1}^{r}W_i.$ Clearly, $W=W_1+W_2+\ldots+W_r.$ Suppose that $z\in W_i\cap \sum\limits_{\substack{l=1\\l\neq i}}^{r}W_l$ for some $i,$ $1\leq i\leq r.$ Then \[z=\sum_{j=1}^{m_i}a_{ij}(y_{ij}^*\otimes x_i)=\sum\limits_{\substack{1\leq l(\neq i)\leq r}}w_l,\text{ where }w_l=\sum\limits_{\substack{1\leq j\leq m_l}}a_{lj}(y_{lj}^*\otimes x_l)\in W_l~,a_{ij}\in \mathbb{R}.\]
    For each $\beta\in\Lambda,$ considering $S_\beta\in \mathbb{L}(\mathbb{X},\mathbb{Y}),$ as defined in (\ref{eq1}), we have, \\ $\sum_{j=1}^{m_i}a_{ij}y_{ij}^*S_{\beta}(x_i)=\sum\limits_{\substack{1\leq l(\neq i)\leq r \\ 1\leq j\leq m_l}}a_{lj}y_{lj}^*S_\beta(x_l)\Rightarrow \sum_{j=1}^{m_i}a_{ij}y_{ij}^*(u_\beta)=0 \Rightarrow a_{ij}=0$ for all $1\leq j \leq m_i.$ Thus, $z=0\Rightarrow W_i\cap \sum\limits_{\substack{l=1\\l\neq i}}^{r}W_l=\{0\}.$ Therefore, $W=\oplus_{i=1}^{r}W_i.$ Hence, $k=dim~W=dim~\oplus_{i=1}^{r}W_i=\oplus_{i=1}^{r}~dim~W_i=m_1+m_2+\ldots+m_r.$ \\
     If $r=n,$ then proceeding similarly, we can show that $k=m_1+m_2+\ldots+m_r.$ This completes the proof of the theorem.
\end{proof}

Using Theorem \ref{th-ind}, we can completely characterize the order of smoothness of a linear operator defined from $\ell_1^n~(n\in \mathbb{R})$ to an arbitrary Banach space.

\begin{corollary}
	Let $\mathbb{Y}$ be an arbitrary Banach space and $T\in \mathbb{L}(\ell_1^n,\mathbb{Y}),\|T\|=1.$ Then $T$ is $k-$smooth if and only if $M_T\cap Ext(B_{\ell_1^n})=\{\pm x_1,\pm x_2,$ $\ldots,\pm x_r\}$ for some $1\leq r\leq n,$ $Tx_i$ is $m_i-$smooth for each $1\leq i\leq r$ and $m_1+m_2+\ldots+m_r=k.$
\end{corollary}
\begin{proof}
The proof easily follows from Theorem \ref{th-ind} and the fact that $B_{\ell_1^n}$ contains only finitely many extreme points and if $M_T\cap Ext(B_{\ell_1^n})=\{\pm x_1,\pm x_2,\ldots,\pm x_r\}$ for some $1\leq r\leq n,$ then $\{x_1,x_2,\ldots,x_r\}$ is always linearly independent set in $\ell_1^n.$
\end{proof}

\begin{remark}
	Note that, if we consider $T\in \mathbb{L}(\ell_\infty^3,\ell_\infty^3)$ defined by $T(x,y,z)=\frac{1}{2}(x+y,x+y,x+y),$ then $M_T\cap Ext(B_{\ell_\infty^3})=\{\pm(1,1,1),\pm(1,1,-1)\}$ and so in this case, we can apply Theorem \ref{th-ind} to conclude that $T$ is $6-$smooth. Whereas if we consider the operator  $T\in \mathbb{L}(\ell_\infty^3,\ell_\infty^3)$ defined by $T(x,y,z)=(x,0,0),$ then $M_T\cap Ext(B_{\ell_\infty^3})=\{\pm(1,1,1),\pm(1,1,-1),\pm(-1,1,1),\pm(1,-1,1)\}$ and so we cannot conclude $k-$smoothness of $T$ from  Theorem \ref{th-ind}.
\end{remark}
If the dimension of $\mathbb{X}$ is infinite then  the Theorem \ref{th-ind} may not be true. To obtain a desired result for infinite dimensional Banach space $\mathbb{X}$, apart from linear independency, we assume additional condition on   $M_T\cap Ext(B_{\mathbb{X}})=\{\pm x_1,\pm x_2,\ldots,\pm x_r\},$ in the form that  $ x_i \bot_B x_j, \forall i,j, i \neq j.$  Note that, in a Banach space $\mathbb{X},$ an element $x$ is Birkhoff-James \cite{B,J} orthogonal to an element $y$, written as,  $ x \bot_B y$  if and only if $ \| x + \lambda y\| \geq \|x\| $ for all scalars $\lambda.$ Although the proof of the following theorem is in the same spirit of the  Theorem \ref{th-ind}, except for  the construction of $S_\beta$, we prove it in details for the convenience of the reader.

\begin{theorem}\label{th-ind2}
	Suppose $\mathbb{X}$ is a smooth, reflexive Banach space and $\mathbb{Y}$ is arbitrary Banach space. Let $\mathbb{K}(\mathbb{X},\mathbb{Y})$ be an $M-$ideal in $\mathbb{L}(\mathbb{X},\mathbb{Y}).$ Suppose that $T\in \mathbb{L}(\mathbb{X},\mathbb{Y}), \|T\|=1$ and  dist$ (T,\mathbb{K}(\mathbb{X},\mathbb{Y})) < 1. $ Suppose that $ M_T\cap Ext(B_{\mathbb{X}})=\{\pm x_1,\pm x_2,\\ \ldots,\pm x_r\},$ where $\{x_1,x_2,\ldots,x_r\}$ is linearly independent in $\mathbb{X}$ and $ x_i \bot_B x_j, \forall i,j, i \neq j.$ Then $T$ is $k-$smooth if and only if $Tx_i$ is $m_i-$smooth for each $1\leq i\leq r$ and $m_1+m_2+\ldots+m_r=k.$
\end{theorem}
\begin{proof}
 Suppose $T$ is $k-$smooth and $Tx_i$ is $m_i-$smooth for each $1\leq i\leq r.$  Then by \cite[Prop. 2.1]{LR}, for each $1\leq i\leq r,$ we have, $	m_i= dim~ span~ Ext~ J(Tx_i).$
Let $\{y_{ij}^*:1\leq j\leq m_i, y_{ij}^*\in Ext ~J(Tx_i)\}$ be a basis of $span~Ext~J(Tx_i)$ for each $1\leq i\leq r.$  Let 
\[W_i=span ~\{y_{ij}^*\otimes x_i:y_{ij}^*\in Ext ~J(Tx_i)\} ~\text{for each } 1\leq i\leq r.\]
We first show that $B_i=\{y_{ij}^*\otimes x_i:1\leq j\leq m_i\}$ is a basis of $W_i.$ 
Let $\sum\limits_{1\leq j\leq m_i}a_j(y_{ij}^*\otimes x_i)=0,$ where $a_j\in \mathbb{R}$ for all $1\leq j\leq m_i.$ Since $\mathbb{X}$ is smooth, for each $1\leq i \leq r,$ there exists a unique hyperspace $H_i$ such that $x_i\perp_B H_i.$ Therefore, $x_j\in H_i$ for all $1\leq j(\neq i)\leq r,$ since $x_i\perp_B x_j$ for all $1\leq j(\neq i)\leq r.$ Consider a Hamel basis $\{u_\beta:\beta \in \Lambda\}$ of $\mathbb{Y}.$ For each $\beta\in\Lambda,$ define $S_\beta:\mathbb{X}\to \mathbb{Y}$ as follows:
\begin{equation} \label{eq2}
\begin{split}
S_\beta x_i&=u_\beta\\
S_\beta x&= 0 \text{~~ for all } x\in H_i. 
\end{split}
\end{equation}
Then it is easy to see that $S_{\beta}\in  \mathbb{L}(\mathbb{X},\mathbb{Y}).$
Now, for each $\beta \in \Lambda,$\\
$\sum\limits_{1\leq j\leq m_i}a_j(y_{ij}^*\otimes x_i)(S_\beta)=0\Rightarrow\sum\limits_{1\leq j\leq m_i}a_jy_{ij}^*S_\beta(x_i)=0\Rightarrow \sum\limits_{1\leq j\leq m_i}a_jy_{ij}^*(u_\beta)=0\Rightarrow \sum\limits_{1\leq j\leq m_i}a_jy_{ij}^*=0\Rightarrow a_j=0 $ for all $1\leq j\leq m_i.$ Thus, $B_i$ is linearly independent. It can be easily verified that $B_i$ is a spanning set of $W_i.$ Hence, $B_i$ is a basis of $W_i$  and so $dim~W_i=m_i$ for each $1\leq i\leq r.$ Now,
\begin{eqnarray*}
	k&=&dim ~span ~J(T)\\
	&=& dim~ span~ Ext ~J(T)\\
	&=& dim~ span ~\{y_{ij}^*\otimes x_i:y_{ij}^*\in Ext ~J(Tx_i),1\leq i\leq r \}\\
	&=& dim~W,\text{~where,}\\
	W&=&span ~\{y_{ij}^*\otimes x_i:y_{ij}^*\in Ext~ J(Tx_i),1\leq i\leq r \}.
\end{eqnarray*}
We now show that $W=\oplus_{i=1}^{r}W_i.$ Clearly, $W=W_1+W_2+\ldots+W_r.$ Suppose that $z\in W_i\cap \sum\limits_{\substack{l=1\\l\neq i}}^{r}W_l$ for some $i,$ $1\leq i\leq r.$ Then \[z=\sum_{j=1}^{m_i}a_{ij}(y_{ij}^*\otimes x_i)=\sum\limits_{\substack{1\leq l(\neq i)\leq r}}w_l,\text{ where }w_l=\sum\limits_{\substack{1\leq j\leq m_l}}a_{lj}(y_{lj}^*\otimes x_l)\in W_l,a_{ij}\in \mathbb{R}.\]
For each $\beta\in\Lambda,$ considering $S_\beta\in \mathbb{L}(\mathbb{X},\mathbb{Y}),$ as defined in (\ref{eq2}), we have, \\ $\sum_{j=1}^{m_i}a_{ij}y_{ij}^*S_{\beta}(x_i)=\sum\limits_{\substack{1\leq l(\neq i)\leq r \\ 1\leq j\leq m_l}}a_{lj}y_{lj}^*S_\beta(x_l)\Rightarrow \sum_{j=1}^{m_i}a_{ij}y_{ij}^*(u_\beta)=0 \Rightarrow a_{ij}=0$ for all $1\leq j \leq m_i.$ Thus, $z=0\Rightarrow W_i\cap \sum\limits_{\substack{l=1\\l\neq i}}^{r}W_l=\{0\}.$ Therefore, $W=\oplus_{i=1}^{r}W_i.$ Hence, $k=dim~W=dim~\oplus_{i=1}^{r}W_i=\oplus_{i=1}^{r}~dim~W_i=m_1+m_2+\ldots+m_r.$ This completes the proof of the theorem.

\end{proof}

\begin{example}
	The above result can be used to determine the order of smoothness of operator $T$ defined on infinite dimensional $\ell_p(1<p(\neq 2)<\infty)$ spaces. As for example consider the operator $T\in \mathbb{L}(\ell_4,\ell_4)$ defined by $$T(a_1,a_2,a_3,a_4,\ldots)=2^{-\frac{3}{4}}(a_1+a_2,a_1-a_2,0,0,\ldots).$$
	Then it is easy to see that $M_T\cap Ext(B_{\ell_4})=\Big\{\pm\Big(\frac{1}{2^{\frac{1}{4}}},\frac{1}{2^{\frac{1}{4}}},0,0,0,\ldots\Big),$     $\pm\Big(-\frac{1}{2^{\frac{1}{4}}},\frac{1}{2^{\frac{1}{4}}},0,0,\ldots\Big)\Big\}.$ Since the space $\ell_4$ and the operator $T$ satisfies all the conditions of Theorem \ref{th-ind2}, we can conclude that $T$ is $2-$smooth.
\end{example}

\section{k-smoothness of operators defined between two-dimensional Banach spaces}
In this section, we completely characterize $k-$smoothness of an operator $T\in \mathbb{L}(\mathbb{X},\mathbb{Y}),$ depending on $|M_T\cap Ext(B_{\mathbb{X}})|,$ when both $\mathbb{X},\mathbb{Y}$ are two-dimensional Banach spaces. Consider the case $|M_T\cap Ext(B_{\mathbb{X}})|=2,$ i.e., $M_T\cap Ext(B_{\mathbb{X}})=\{\pm x_1\},$ in this case $T$ is smooth if $Tx_1$ is smooth and $T$ is $2-$smooth if $Tx_1$ is non-smooth, which follows clearly from Theorem \ref{th-ind}. Next, consider the case $|M_T\cap Ext(B_{\mathbb{X}})|=4,$ i.e., $M_T\cap Ext(B_{\mathbb{X}})=\{\pm x_1,\pm x_2\},$ in this case following Theorem \ref{th-ind}, we can conclude that $T$ is $2-$smooth when both $Tx_1,Tx_2$ are smooth, $T$ is $3-$smooth when only one of $Tx_1,Tx_2$ is smooth and $T$ is $4-$smooth when both $Tx_1,Tx_2$ are non-smooth. In case $|M_T\cap Ext(B_{\mathbb{X}})|>4,$ the situation is little bit complicated and we have to consider the two cases: $|M_T\cap Ext(B_{\mathbb{X}})|=6$ and $|M_T\cap Ext(B_{\mathbb{X}})|\geq 8.$ We first prove the following theorem.

\begin{theorem}\label{th-mt6}
	Suppose $\mathbb{X}, \mathbb{Y}$ are two-dimensional Banach spaces and $T\in \mathbb{L}(\mathbb{X},\mathbb{Y})$ is such that $\|T\|=1$ and $M_T\cap Ext(B_{\mathbb{X}})=\{\pm x_1,\pm x_2,\pm x_3\}.$ Then the following holds:\\
	(i) If $Tx_i$ is smooth for each $1\leq i\leq 3,$ then $T$ is $3-$smooth.\\
	(ii) If $Tx_1$ is not smooth and either $Tx_2,Tx_3$ are interior point of same line segment of unit sphere or $Tx_2,-Tx_3$ are interior point of same line segment of unit sphere, then $T$ is $3-$smooth.\\
	(iii) If $Tx_1$ is not smooth, $Tx_2,Tx_3$ are not interior point of the same line segment of unit sphere and $Tx_2,-Tx_3$ are not interior point of the same line segment of unit sphere, then $T$ is $4-$smooth. 
\end{theorem}
\begin{proof}
	Clearly, $T$ is $k-$smooth for some $1\leq k\leq 4,$ since $dim(\mathbb{X})=dim(\mathbb{Y})=2.$
	$(i)$ Suppose $Tx_i$ is smooth for each $1\leq i\leq 3.$ Then $Tx_i$ has unique supporting linear functional for each $1\leq i\leq 3.$ We first show that $Tx_1,Tx_2,Tx_3$ cannot have same supporting linear functional. If possible, suppose that $J(Tx_i)=\{y^*\}$ for all $i=1,2,3.$ Then $y^*(Tx_1)=y^*(Tx_2)=y^*(Tx_3)=1.$ Hence, for all $t\in[0,1],~y^*(tTx_1+(1-t)Tx_2)=1\Rightarrow \|tTx_1+(1-t)Tx_2\|=1,$ since $\|y^*\|=1.$ Thus, $\|T(tx_1+(1-t)x_2)\|=1$ and $\|T\|=1$ together gives that $\|tx_1+(1-t)x_2\|=1$ for all $t\in[0,1].$ This implies that $x_1,x_2$ are on same line segment of unit sphere. Similarly, $x_1,x_3$ and $x_2,x_3$ are on same line segment of unit sphere. This contradicts that $x_1,x_2,x_3$ are distinct extreme points of $B_{\mathbb{X}}.$ Therefore, without loss of generality, we may assume that $J(Tx_i)=\{y_i^*\}$ for all $i=1,2,3$ and $y_1^*\neq \pm y_2^*.$ Since $\mathbb{X}$ is two dimensional and $x_1,x_2,x_3$ are distinct extreme points of $B_{\mathbb{X}},$ we have $x_3=\gamma x_1+\delta x_2$ for some $\gamma(\neq 0),\delta(\neq 0)\in \mathbb{R}.$ Now, $y_1^*\neq \pm y_2^*\Rightarrow \{y_1^*,y_2^*\}$ is linearly independent in $Y^*.$ Therefore, $y_3^*=\alpha y_1^*+\beta y_2^*$ for some $\alpha,\beta\in \mathbb{R}.$  Since $T$ is $k-$smooth, 
		\begin{eqnarray*}
		k&=&dim ~span ~J(T)\\
		&=& dim~ span~ Ext ~J(T)\\
		&=& dim~ span ~\{y_i^*\otimes x_i:1\leq i\leq 3 \}.
	\end{eqnarray*}
We show that $\{y_i^*\otimes x_i:1\leq i\leq 3 \}$ is linearly independent. Let
\begin{eqnarray*}
	&&a_1y_1^*\otimes x_1+a_2y_2^*\otimes x_2+a_3y_3^*\otimes x_3=0,~\text{where~} a_1,a_2,a_3\in \mathbb{R}, \\
	&\Rightarrow &	a_1y_1^*\otimes x_1+a_2y_2^*\otimes x_2+a_3(\alpha y_1^*+\beta y_2^*)\otimes(\gamma x_1+\delta x_2)=0\\
	&\Rightarrow& (a_1+a_3\alpha\gamma)y_1^*\otimes x_1+(a_2+a_3\beta\delta)y_2^*\otimes x_2+a_3\alpha\delta y_1^*\otimes x_2+a_3\beta\gamma y_2^*\otimes x_1=0.
\end{eqnarray*}
Now, using Lemma \ref{lemma-01}, we have, $a_1+a_3\alpha\gamma=0,~a_2+a_3\beta\delta=0,~a_3\alpha\delta=0$ and $a_3\beta\gamma=0.$ Solving these $4$ equations, we get $a_1=a_2=a_3=0.$ Therefore,  $\{y_i^*\otimes x_i:1\leq i\leq 3 \}$ is linearly independent. Thus, $T$ is $3-$smooth. \\

$(ii)$ Suppose that $Tx_1$ is not smooth. Without loss of generality, assume that  $Tx_2,Tx_3$ are interior point of same line segment of unit sphere. Then $Tx_2,Tx_3$ have same unique supporting linear functional say, $z^*,$ i.e., $J(Tx_2)=J(Tx_3)=\{z^*\}.$ Since $Tx_1$ is not smooth and $\mathbb{Y}$ is two-dimensional, it is easy to see that $Ext ~J(Tx_1)=\{y_1^*,y_2^*\}$ for some linearly independent set $\{y_1^*,y_2^*\}$  of $\mathbb{Y}^*.$  Now, $x_3=ax_1+bx_2$ for some $a(\neq 0),b(\neq 0)\in \mathbb{R}$ and $z^*=\alpha y_1^*+\beta y_2^*$ for some $\alpha,\beta\in \mathbb{R}.$ Therefore, $z^*\otimes x_3=(\alpha y_1^*+\beta y_2^*)\otimes(ax_1+bx_2)=a\alpha y_1^*\otimes x_1+a\beta y_2^*\otimes x_1+bz^*\otimes x_2\in span \{y_1^*\otimes x_1,y_2^*\otimes x_1,z^*\otimes x_2\}.$ Thus, 
\begin{eqnarray*}
k&=&dim~span~Ext~J(T)\\
&=&dim~span~\{y_1^*\otimes x_1,y_2^*\otimes x_1,z^*\otimes x_2,z^*\otimes x_3\}\\
&=&dim~span~\{y_1^*\otimes x_1,y_2^*\otimes x_1,z^*\otimes x_2\}.
\end{eqnarray*}
We next show that $\{y_1^*\otimes x_1,y_2^*\otimes x_1,z^*\otimes x_2\}$ is linearly independent. Let $a_1y_1^*\otimes x_1+a_2y_2^*\otimes x_1+a_3z^*\otimes x_2=0,$ where $a_i\in \mathbb{R}~(i=1,2,3).$ Then 
\begin{eqnarray}\label{eq-01}
a_1y_1^*S( x_1)+a_2y_2^*S( x_1)+a_3z^*S( x_2)=0~\text{for all} ~S\in \mathbb{L}(\mathbb{X},\mathbb{Y}).
\end{eqnarray}
 Define $S_1,S_2\in \mathbb{L}(\mathbb{X},\mathbb{Y})$ as follows:
\begin{align*}
S_1x_1 & = 0 &    S_2 x_1 & = u_2\\
S_1x_2 & = u_1  &   S_2x_2 & = 0,
\end{align*}
where $u_1\notin \ker(z^*)$ and $u_2\in \ker(y_1^*)\setminus \ker(y_2^*).$ Now, putting $S_1,S_2$ in (\ref{eq-01}), we get,  $a_2=a_3=0.$ Thus, $a_1y_1^*\otimes x_1=0.$ Since $x_1\neq 0$ and $y_1^*\neq 0,$ we have, $a_1=0.$ Therefore, $\{y_1^*\otimes x_1,y_2^*\otimes x_1,z^*\otimes x_2\}$ is linearly independent subset of $\mathbb{L}(\mathbb{X},\mathbb{Y})^*.$ Thus, $k=3$ and so $T$ is $3-$smooth.\\

$(iii)$ Suppose $Tx_1$ is not smooth, $Tx_2,Tx_3$ are not interior point of the same line segment of unit sphere and $Tx_2,-Tx_3$ are not interior point of the same line segment of unit sphere. Then $Ext~J(Tx_1)=\{y_{11}^*,y_{12}^*\}$ for some linearly independent subset $\{y_{11}^*,y_{12}^*\}$ of $\mathbb{Y}^*$ and there exist $y_2^*\in Ext~J(Tx_2)$ and $y_3^*\in Ext~J(Tx_3)$ such that $y_2^*\neq \pm y_3^*.$ Now,
\begin{eqnarray*}
	4\geq k&=&dim~span~Ext~J(T)\\
	&\geq&dim~span~\{y_{11}^*\otimes x_1,y_{12}^*\otimes x_1,y_2^*\otimes x_2,y_3^*\otimes x_3\}.
\end{eqnarray*}
As before, choosing $S$ suitably from $\mathbb{L}(\mathbb{X},\mathbb{Y})$ it can be easily shown that $\{y_{11}^*\otimes x_1,y_{12}^*\otimes x_1,y_2^*\otimes x_2,y_3^*\otimes x_3\}$ is linearly independent subset of $\mathbb{L}(\mathbb{X},\mathbb{Y})^*.$ Thus, $k=4$ and so $T$ is $4-$smooth. This completes the proof of the theorem.
\end{proof}

In addition to $|M_T\cap Ext(B_{\mathbb{X}})|=6,$ if we assume the strict convexity of either $\mathbb{X}$ or $\mathbb{Y},$ then the $k-$smoothness of $T$ can be characterized as follows.

\begin{corollary}
Suppose	$\mathbb{X},\mathbb{Y}$ are two-dimensional Banach spaces and either $\mathbb{X}$ or $\mathbb{Y}$ is strictly convex. Let $T\in \mathbb{L}(\mathbb{X},\mathbb{Y})$ be such that $M_T\cap Ext(B_{\mathbb{X}})=\{\pm x_1,\pm x_2,\pm x_3\}.$ Then $T$ is $3-$smooth if and only if $Tx_i$ is smooth for all $i=1,2,3,$ otherwise $T$ is $4-$smooth.
\end{corollary}
\begin{proof}
	At first suppose that $\mathbb{X}$ is strictly convex. We only show that case $(ii)$ of Theorem \ref{th-mt6} does not hold. If possible, suppose that $Tx_2,Tx_3$ are interior point of same line segment. Then $Tx_2,Tx_3$ have same supporting linear functional.Then there exists $y^*\in S_{\mathbb{Y}^*}$ such that $y^*(Tx_2)=y^*(Tx_3)=1.$ So for all $t\in[0,1],y^*((1-t)Tx_2+tTx_3)=1\Rightarrow\|(1-t)x_2+tx_3\|=1$ which contradicts that $\mathbb{X}$ is strictly convex. Therefore, case $(ii)$ of Theorem \ref{th-mt6} does not hold and the result follows from Theorem \ref{th-mt6}.\\
When $\mathbb{Y}$ is strictly convex, case $(ii)$ of Theorem \ref{th-mt6} does not arise and the result follows easily.
\end{proof}

The only case remaining to completely characterize $k-$smoothness of an operator $T$ between two-dimensional Banach spaces $\mathbb{X}$ and $\mathbb{Y}$ is $|M_T\cap Ext(B_{\mathbb{X}})|\geq 8.$ In the next theorem, we consider this case.

\begin{theorem}\label{th-mt8}
Suppose	$\mathbb{X},\mathbb{Y}$ are two-dimensional Banach spaces. Let $T\in \mathbb{L}(\mathbb{X},\mathbb{Y})$ be such that $|M_T\cap Ext(B_{\mathbb{X}})|\geq 8.$  Then the following holds:\\
(i) If $Tx$ is not smooth for some $x\in M_T\cap Ext(B_{\mathbb{X}}),$ then $T$ is $4-$smooth.\\
(ii)  Suppose $Tx$ is smooth for each $x\in M_T\cap Ext(B_{\mathbb{X}}).$ If there exist $x_i\in  M_T\cap Ext(B_{\mathbb{X}}),~y_i^*\in J(Tx_i)$ for $i=1,2,3,4$ such that $x_2=ax_1+bx_3, x_4=cx_1+dx_3$ and $y_2^*=\alpha_1y_1^*+\alpha_2y_3^*,y_4^*=\beta_1y_1^*+\beta_2y_3^*$ with $ \beta_1\alpha_2 ad-\beta_2\alpha_1bc\neq 0,$ then $T$ is $4-$smooth. Otherwise $T$ is $3-$smooth.
\end{theorem}
\begin{proof}
    Clearly, $T$ is $k-$smooth for some $1\leq k\leq 4,$ since $dim(\mathbb{X})=dim(\mathbb{Y})=2.$ Since $|M_T\cap Ext(B_{\mathbb{X}})|\geq 8,$ we may assume that $\{\pm x_1,\pm x_2,\pm x_3,\pm x_4\}\subseteq M_T\cap Ext(B_{\mathbb{X}}).$ \\
 	$(i)$ Assume that $Tx_1$ is not smooth. Without loss of generality, we may assume that $x_1=\frac{(1-s)x_2-sx_4}{\|(1-s)x_2-sx_4\|}$ and $x_3=\frac{(1-t)x_2+tx_4}{\|(1-t)x_2+tx_4\|}$ for some $s,t\in(0,1).$ Let $y_{11}^*,y_{12}^*$ be two linearly independent vectors in $Ext~J(Tx_1).$ Suppose $y_2^*\in Ext~J(Tx_2),y_4^*\in Ext~J(Tx_4).$ Then $y_2^*\neq \pm y_4^*,$ for if $y_2^*= y_4^*,$ then as in Theorem \ref{th-mt6} $(i)$, it can be shown that $\|(1-t)x_2+tx_4\|=1$ for all $t\in[0,1].$ This contradicts that $x_3$ is an extreme point of $B_{\mathbb{X}}.$ Thus, $y_2^*\neq y_4^*.$ Similarly, $y_2^*\neq -y_4^*.$ Thus, $y_2^*$ and $y_4^*$ are linearly independent. Since $T$ is $k-$smooth, we have,
	\begin{eqnarray*}
	4 \geq k&=& dim~span~Ext~J(T)\\
	&\geq & dim ~span~\{y_{11}^*\otimes x_1,y_{12}^*\otimes x_1,y_2^*\otimes x_2,y_4^*\otimes x_4\} .
	\end{eqnarray*}
We claim that $\{y_{11}^*\otimes x_1,y_{12}^*\otimes x_1,y_2^*\otimes x_2,y_4^*\otimes x_4\}$ is linearly independent. Let $a y_{11}^*\otimes x_1+b y_{12}^*\otimes x_1+c y_2^*\otimes x_2+d y_4^*\otimes x_4=0,$ where  $a,b,c,d\in \mathbb{R}.$ Then 
\begin{eqnarray}\label{eq-02}
a y_{11}^*S(x_1)+b y_{12}^*S(x_1)+c y_2^*S( x_2)+d y_4^*S( x_4)=0 ~\forall~S\in \mathbb{L}(\mathbb{X},\mathbb{Y}).
\end{eqnarray}
For $1\leq i\leq 4,$ define $S_i\in \mathbb{L}(\mathbb{X},\mathbb{Y})$ as follows:
\begin{align*}
S_1x_1 & = 0 &    S_2 x_1 & = 0 & S_3 x_1&=u_3 &  S_4 x_1=u_4\\
S_1x_2 & = u_1  &   S_2x_2 & = u_2 & S_3 x_2&=0 &  S_4 x_2=0,
\end{align*}
where $u_1\in \ker(y_2^*) \setminus \ker(y_4^*)$ and $u_2\in \ker(y_4^*)\setminus \ker(y_2^*),u_3\in  \ker(y_{11}^*) \setminus \ker(y_{12}^*),u_4\in  \ker(y_{12}^*) \setminus \ker(y_{11}^*). $ Now, putting $S_1,S_2,S_3,S_4$ in (\ref{eq-02}), we get, $a=b=c=d=0.$ Therefore, $\{y_{11}^*\otimes x_1,y_{12}^*\otimes x_1,y_2^*\otimes x_2,y_4^*\otimes x_4\}$ is linearly independent. Thus, $k=4$ and so $T$ is $4-$smooth.\\

$(ii)$  Suppose $Tx $ is smooth for each $x \in M_T\cap Ext(B_{\mathbb{X}}) $ and  $\beta_1\alpha_2 ad-\beta_2\alpha_1bc\neq 0.$ Clearly  
	$ 	4 \geq k= dim~span~Ext~J(T)\geq  dim ~span~\{y_i^*\otimes x_i:1\leq i\leq 4\} $.
We claim that $ \{y_i^*\otimes x_i:1\leq i\leq 4\} $ is linearly independent.\\
Let 
$a_1y_1^*\otimes x_1+a_2y_2^*\otimes x_2+a_3y_3^*\otimes x_3+a_4y_4^*\otimes x_4=0, $ where $a_i\in \mathbb{R}, 1 \leq i \leq4.$  Then  \\
$ a_1y_1^*\otimes x_1+a_2(\alpha_1y_1^*+\alpha_2y_3^*)\otimes (ax_1+bx_3)+
a_3y_3^*\otimes x_3+a_4(\beta_1y_1^*+\beta_2y_3^*)\otimes (cx_1+dx_3)=0$.\\
$\Rightarrow (a_1+a_2\alpha_1 a+a_4\beta_1 c)y_1^*\otimes x_1+(a_2\alpha_1 b+a_4\beta_1 d)y_1^*\otimes x_3+
(a_2\alpha_2 a+a_4\beta_2 c)y_3^*\otimes x_1+(a_3+a_2\alpha_2 b+a_4\beta_2 d)y_3^*\otimes x_3=0.$

Now, using Lemma \ref{lemma-01}, $\{y_1^*\otimes x_1,y_1^*\otimes x_3,y_3^*\otimes x_1,y_3^*\otimes x_3\}$ is a linearly independent set. Hence, $a_1+a_2\alpha_1 a+a_4\beta_1 c=0,a_2\alpha_1 b+a_4\beta_1 d=0,a_2\alpha_2 a+a_4\beta_2 c=0$ and $a_3+a_2\alpha_2 b+a_4\beta_2 d=0.$ Solving these equations, we get, $a_1,a_2,a_3,a_4=0,$ since $\beta_1\alpha_2 ad-\beta_2\alpha_1bc\neq 0.$ Therefore, $dim ~span~\{y_i^*\otimes x_i:1\leq i\leq 4\}=4\Rightarrow k=4.$ Thus, $T$ is $4-$smooth.\\
Now, suppose that for each $\{\pm x_i:1\leq i\leq 4\}\subseteq   M_T\cap Ext(B_{\mathbb{X}}) $ and $y_i^*\in J(Tx_i)$ for $i=1,2,3,4,$  $x_2=ax_1+bx_3, x_4=cx_1+dx_3$ and $y_2^*=\alpha_1y_1^*+\alpha_2y_3^*,y_4^*=\beta_1y_1^*+\beta_2y_3^*\Rightarrow \beta_1\alpha_2 ad-\beta_2\alpha_1bc= 0.$ Then $\{y_i^*\otimes x_i:1\leq i\leq 4\}$ is a linearly dependent set. Hence, $k<4.$ Proceeding similarly as in Theorem \ref{th-mt6} $(i)$ we can show that $\{y_i^*\otimes x_i:1\leq i\leq 3\}$ is linearly independent. Therefore, $k=3$ and so $T$ is $3-$smooth. This completes the proof of the theorem.
\end{proof}

Observe that if $\mathbb{X}$ is a two-dimensional Banach space such that the unit sphere of $\mathbb{X}$ is a polygon with more than $6$ vertices, then the identity operator on $\mathbb{X}$ satisfies the hypothesis of Theorem \ref{th-mt8} $(i)$ and so it is $4-$smooth. Now, we exhibit two examples to show that there exist two-dimensional Banach spaces $\mathbb{X},\mathbb{Y}$ and operators $T\in \mathbb{L}(\mathbb{X},\mathbb{Y})$ such that both the cases of Theorem \ref{th-mt8} $(ii)$ hold.

\begin{example}
	(i) Suppose $\mathbb{X}$ is a two-dimensional Banach space such that the unit sphere of  $\mathbb{X}$ is a regular octagon with vertices $\pm (1,0),\pm (\frac{1}{\sqrt{2}},\frac{1}{\sqrt{2}}), \pm (0,1), \pm (-\frac{1}{\sqrt{2}},\frac{1}{\sqrt{2}}).$ Define $T\in \mathbb{L}(\mathbb{X},\mathbb{X})$ by $T(1,0)=(\frac{1}{2}+\frac{1}{2\sqrt{2}},\frac{1}{2\sqrt{2}}),$ $T(0,1)=(-\frac{1}{2\sqrt{2}},\frac{1}{2}+\frac{1}{2\sqrt{2}}).$ Then $M_T\cap Ext(B_{\mathbb{X}})=\{\pm (1,0),\pm (\frac{1}{\sqrt{2}},\frac{1}{\sqrt{2}}), \pm (0,1), \pm (-\frac{1}{\sqrt{2}},\frac{1}{\sqrt{2}})\}$ and $Tx$ is smooth for each $x\in M_T\cap Ext(B_{\mathbb{X}}).$ In this case, it can be verified that $T$  is $3-$smooth.\\
	(ii) Suppose that $\mathbb{X},\mathbb{Y}$ are two-dimensional Banach spaces such that $S_{\mathbb{X}}$ is a regular octagon with vertices $\pm (1,0),\pm (\frac{1}{\sqrt{2}},\frac{1}{\sqrt{2}}), \pm (0,1), \pm (-\frac{1}{\sqrt{2}},\frac{1}{\sqrt{2}})$ and $S_{\mathbb{Y}}$ is an irregular octagon with vertices $\pm (1,0),\pm \Big(\frac{17\sqrt{2}-30}{324-234\sqrt{2}},\frac{35\sqrt{2}-56}{324-234\sqrt{2}}\Big), \pm (0,1), \pm (-\frac{1}{\sqrt{2}},\frac{1}{\sqrt{2}}).$ Define $T\in \mathbb{L}(\mathbb{X},\mathbb{Y})$ by $T(1,0)=(\frac{5\sqrt{2}+4}{12},\frac{2+3\sqrt{2}}{12}),T(0,1)=(-\frac{\sqrt{2}}{4},\frac{2+\sqrt{2}}{4}).$ Then  $M_T\cap Ext(B_{\mathbb{X}})=\{\pm (1,0),\pm (\frac{1}{\sqrt{2}},\frac{1}{\sqrt{2}}), \pm (0,1), \pm (-\frac{1}{\sqrt{2}},\frac{1}{\sqrt{2}})\}$  and $Tx$ is smooth for each $x\in M_T\cap Ext(B_{\mathbb{X}}).$ In this case, it can be verified that $T$  is $4-$smooth.\\
\end{example}

In \cite[Th. 4.2]{W},  W\'{o}jcik proved that in an $n-$dimensional Banach space $\mathbb{X},$ if an unit vector $x\in \mathbb{X}$ is $n-$smooth, then $x$ is an exposed point. In the following theorem, we prove the converse of \cite[Th. 4.2]{W} for polyhedral Banach space.

\begin{theorem}
	Let $\mathbb{X}$ be an $n-$dimensional polyhedral Banach space. If $x\in S_{\mathbb{X}}$ is an exposed point of $\mathbb{X},$ then $x$ is $n-$smooth.
\end{theorem}
\begin{proof}
	Suppose $x\in S_{\mathbb{X}}$ is an exposed point of $\mathbb{X}$ and $x$ is $k-$smooth. If possible, suppose that $k<n.$ Let $\{x_1^*,x_2^*,\ldots,x_k^*\}$ be linearly independent subset of $Ext~J(x).$ It is easy to see that $dim(\ker x_1^*\cap \ker x_2^*\cap\ldots\cap \ker x_k^*)=n-k>0.$ Suppose $z\in \cap_{i=1}^k\ker x_i^*.$ Let $Y=span\{x,z\}.$ Then $Y$ is a polygonal Banach space. If possible, suppose that $x$ is $2-$smooth in $Y.$ Then there exist linearly independent vectors $y_1^*,y_2^*\in S_{Y^*}$ such that $y_1^*(x)=y_2^*(x)=1.$ Let $z_1^*,z_2^*$ be two norm preserving extensions of $y_1^*$  and $y_2^*$ respectively. Then $z_1^*,z_2^*\in J(x).$ Thus, $z_1^*,z_2^*\in span ~J(x)=span~Ext~J(x).$ Since $x_i^*(z)=0$ for all $1\leq i\leq k,$ $z_1^*(z)=z_2^*(z)=0.$ Hence, $y_1^*(z)=y_2^*(z)=0,$ contradicting that $y_1^*,y_2^*$ are linearly independent. This proves that $x$ is smooth point in $Y.$ Hence, there exist $x_1,x_2\in S_Y\subseteq S_{\mathbb{X}}$ such that $x=\frac{1}{2}x_1+\frac{1}{2}x_2.$ Thus, $x$ is not an extreme point of $B_{\mathbb{X}}$ and so $x$ is not an exposed point of $B_{\mathbb{X}},$ contradicting the hypothesis of the theorem. Therefore, $k=n.$ This completes the proof of the theorem. 
\end{proof}

\bibliographystyle{amsplain}

\end{document}